\documentclass{amsproc}
\usepackage{amsmath,amssymb}

\DeclareMathAlphabet{\mathbbold}{U}{bbold}{m}{n}

\def\EC{\mathcal{E}}

\newcommand{\cP}{\mathcal{P}}
\def\R{\mathbb{R}}
\def\N{\mathbb{N}}
\def\Z{\mathbb{Z}}
\newcommand{\lb}{\ell_*}

\def\pca{\operatorname{\mathbf P}}
\def\pcaf{\pca_{\mathrm F}}
\def\pcai{\pca_{\mathrm I}}
\def\larc{\operatorname{lp}}
\def\earc{\operatorname{ep}}
\def\sat{\operatorname{Sat}}
\def\vert{\operatorname{V}}

\def\NCA{{\rm (NC)}}
\def\ANZ{{\rm (ZC)}}
\def\ATC{{\rm (TC)}}
\def\AC{{\rm (ZC,TC)}}
\def\NEW#1{{\em #1}}
\newcommand{\rbar}{\overline{\R}}
\newcommand{\rmax}{\mathbb{R}_{\max}}
\def\Bd{{\mathcal B}}
\newcommand{\RBX}{\rbar^{_{\scriptstyle X}}}
\def\KK{\mathcal{K}}
\newcommand{\set}[2]{\{#1\mid #2\}}

\newtheorem{theorem}{Theorem}[section]
\newtheorem{proposition}[theorem]{Proposition}

\newtheorem{pty}[theorem]{Property}
\newtheorem{corollary}[theorem]{Corollary}

\theoremstyle{definition}
\newtheorem{ex}[theorem]{Example}
\newtheorem{definition}[theorem]{Definition}

\theoremstyle{remark}
\newtheorem{remark}[theorem]{Remark}

\newcounter{myenumerate}
\renewcommand{\themyenumerate}{(\roman{myenumerate})}
\newenvironment{myenumerate}{\begin{list}{\themyenumerate }{
\usecounter{myenumerate}\setlength{\labelsep}{0.5ex}
\setlength{\leftmargin}{0pt}\setlength{\labelwidth}{-\labelsep}
}}{\end{list}}

\newenvironment{myitemize}{\begin{list}{}{
\setlength{\labelsep}{0.5ex}
\setlength{\leftmargin}{0pt}\setlength{\labelwidth}{-\labelsep}
}}{\end{list}}

\numberwithin{equation}{section}

\begin{document}

\title[The Optimal Assignment Problem for a Countable State Space]{The Optimal Assignment Problem \\ for a Countable State Space}
\author{Marianne Akian}
\address{Marianne Akian, INRIA, Saclay--\^Ile-de-France, and CMAP, Ecole
Polytechnique, Route de Saclay, 91128 Palaiseau Cedex, France}
\email{Marianne.Akian@inria.fr}
\author{St\'ephane Gaubert}
\address{St\'ephane Gaubert, INRIA, Saclay--\^Ile-de-France, and CMAP, Ecole
Polytechnique, Route de Saclay, 91128 Palaiseau Cedex, France}
\email{Stephane.Gaubert@inria.fr}
\author{Vassili Kolokoltsov}
\address{Vassili Kolokoltsov, Department of Statistics, University of Warwick, Coventry CV4 7AL, UK}
\email{v.Kolokoltsov@warwick.ac.uk}
\subjclass[2000]{Primary 90B80; Secondary 39B42, 90C08, 90C39}




\date{\today}
\begin{abstract}
Given a $n\times n$ matrix $B=(b_{ij})$ with real entries, the optimal
assignment problem is to find a permutation $\sigma$ of
$\{1,\ldots,n\}$ maximising the sum $\sum_{i=1}^n b_{i\sigma(i)}$.  In
discrete optimal control and in the theory of discrete event systems,
one often encounters the problem of solving the equation $Bf=g$ for a
given vector $g$, where the same symbol $B$ denotes the corresponding
max-plus linear operator, $(Bf)_i:=\max_{1\leq j\leq
n}b_{ij}+f_j$. The matrix $B$ is said to be strongly regular when
there exists a vector $g$ such that the equation $Bf=g$ has a unique
solution $f$.  A result of Butkovi\v{c} and Hevery shows that $B$ is
strongly regular if and only if the associated optimal assignment
problem has a unique solution. We establish here an extension of this
result which applies to max-plus linear operators over a countable
state space.  The proofs use the theory developed in a previous work
in which we characterised the unique solvability of equations
involving Moreau conjugacies over an infinite state space, in terms of
the minimality of certain coverings of the state space by generalised
subdifferentials. 
\end{abstract}
\thanks{The two first authors were partially supported by the joint RFBR-CNRS grant 05-01-02807.}
\maketitle

\section{Introduction}

Let $B=(b_{ij})$ be a $n\times n$ matrix with real entries.
The optimal assignment problem
is to find a permutation $\sigma$ of $\{1,\ldots,n\}$ maximising the sum $\sum_{i=1}^n b_{i\sigma(i)}$.

This problem can be interpreted algebraically by introducing the
max-plus or tropical semiring, $\rmax$,
which is the set $\R\cup\{-\infty\}$, where $\R$ is the set of real numbers, 
equipped with the
addition $(a,b)\mapsto \max(a,b)$ and the multiplication $(a,b)\mapsto a+b$.
With these operations, one can define the notions of vectors, matrices, linear
operators. 
In particular, the value of the optimal assignment 
is
nothing but the permanent of the matrix $B$, evaluated in the semiring
$\rmax$.

We also associate to the matrix $B$ a linear operator over the max-plus semiring, which sends the vector $f\in\rmax^n$,
to the vector  $Bf\in \rmax^n$ given by 
$(Bf)_i:=\max_{1\leq j\leq n}b_{ij}+f_j$
(here we keep the usual notations $\max$ and $+$ for scalars, but use the
linear operator notation $Bf$ instead of a non linear one like $B(f)$).
The map $f\mapsto B(-f)$ is a special case of \NEW{Moreau conjugacy}, 
see~\cite[Chapter~11, Section~E]{rockwets}, \cite{Si}, \cite{AGK1,AGK2}.

Butkovi\v{c} and Hevery~\cite{BH} found a remarkable
relation between the equation $Bf=g$ and the optimal assignment
problem. They defined a matrix $B$ with finite real entries to be \NEW{strongly regular} when there exists a vector $g\in \R^n$
such that the equation $Bf=g$ has a unique solution $f\in \R^n$. 
They showed that $B$ is strongly regular if and only if the associated
optimal assignment problem has a unique solution. 
Further properties of strongly regular matrices appeared
in~\cite{butkovip94,Bu}. In particular, the matrix $B$ is strongly
regular if and only if the space generated by its
columns is of nonempty  interior. 

The same notion arose later on in the work of
Richter-Gebert, Sturmfels, and Theobald~\cite{RGST},
who defined a matrix to be \NEW{tropically singular} if
its columns are not in ``generic position'' in the tropical
sense, meaning that they are included in the tropical analogue of
a hyperplane. 
They showed that a (square) matrix is tropically nonsingular if and only if 
the associated optimal assignment problem has a unique solution. So
tropical nonsingularity and strong regularity coincide.

The infinite dimensional version of the optimal assignment
problem is nothing but the celebrated Monge-Kantorovich mass transportation
problem. The equation $Bf=g$ is a well known tool
in the study of this problem via the infinite dimensional linear programming formulation introduced by Kantorovitch.
Indeed, a feasible solution of the dual problem
of this linear programming problem
consists precisely (up to a change of sign) of a pair of functions
$f,g$ such that $Bf\leq g$, and when $f$ and $g$ are optimal,
a complementary slackness property shows, at least formally,
that $Bf=g$. This motivates the search of infinite dimensional
analogues of the theorem of Butkovi\v{c} and Hevery. 

The cases in which the state space is non compact can be
regarded as degenerate. In this paper, we consider
the simplest among these cases: we 
study the optimal assignment problem over a
denumerable state space. 

Loosely speaking, this problem aims at finding the optimal marriages
in a society with a denumerable number of boys and girls. The interest
in these questions goes back to the very origin of matching theory, since
infinite graphs were already considered in K\"onig's book~\cite{konig}.
The theory of matching in infinite graphs has been considerably developed
after K\"onig, we refer the reader
to the survey of Aharoni~\cite{aharoni}, in which generalisations
of fundamental results in matching theory, like
K\"onig's theorem, Hall's marriage theorem, or Birkhoff's
theorem on bistochastic matrices, can be found.

In this paper, we extend the theorem
of Butkovi\v{c} and Hevery to the denumerable
setting, under some critical technical assumptions.

Our approach relies on the characterisation of the
existence and of the uniqueness of the solution
of the equation $Bf=g$ in terms of covering
by generalised subdifferentials given in our previous work~\cite{AGK1,AGK2}.

This characterisation originates from a result 
of Vorobyev~\cite[Theorem  2.6]{Vo}, who dealt with
a finite state space and introduced
a notion of ``minimal resolvent coverings'' of $X$.
Vorobyev's approach
was systematically developed by Zimmermann~\cite[Chapter 3]{Zi.K}, who 
considered several algebraic structures and 
allowed in particular the matrix $B$ to have $-\infty$ entries.
%
%
The sets arising in Vorobyev's covering were shown to
be special cases of subdifferentials in~\cite{AGK1,AGK2}, leading
to an extension of Vorobyev's theorem to Moreau conjugacies 
and even to the more general
case of ``functional Galois connections''.
The existence and uniqueness results proved there contain as special cases
Vorobyev's combinatorial result, and some properties of convex analysis
(for instance, that 
an essentially smooth lower semicontinuous proper convex function
on $\R^n$ has a unique preimage by the Fenchel transform).

In the characterisation of the existence and uniqueness
of the solution of $Bf=g$ in~\cite{AGK1,AGK2}, some
mild compactness assumptions are needed. 
These assumptions lead us here to require a tightness
condition on the kernel, see Assumption~\ATC\ below. The latter
is of the same nature as the tightness condition used by
Akian, Gaubert and Walsh~\cite{AGW} in denumerable
max-plus spectral theory.

We also note that in the denumerable case, the value of the permanent may be ill
defined, because the weight of a permutation is 
the sum of a possibly divergent series. However, the optimality
of a permutation can be expressed in full generality, because
the difference of weights of permutations make sense under general
circumstances, see Definition~\ref{defi-assign}. This definition
is somehow reminiscent of the treatment of ``infinite extremals'' in dynamic programming, see~\cite{KM} for more background on this topic.

%
%


After a preliminary section introducing 
the notations and motivating the main assumptions, we
formulate our main results in Section~\ref{sec-form}
as Theorems~\ref{th1},~\ref{th2} and~\ref{th3}, and prove them
in Sections~\ref{sec2} and~\ref{sec3}. 

Let us conclude this introduction by listing further references.
Motivations to consider Moreau conjugacies or max-plus linear operators
with kernels can be found in~\cite{Vo,CG,maslov73,gondranminouxbook,BCOQ,ccggq99,duality,KM,Gu,LMS,LM,mceneaney}. 
Recent development are highly influenced by tropical geometry via
the so-called dequantisation procedure~\cite{LM,itenberg}.
The Moreau conjugacies, or equivalently, the max-plus linear
operators with kernel considered here,
are the most natural $(\max,+)$-linear operators,
though they do not exhaust all of them (see e.g.\ \cite{Ak}, \cite{Ko}, \cite{LMS}, \cite{LS} and the references therein for classical and recent results on
``kernel type'' representations).
More insight on the notion of tropical singularity is given in the survey~\cite{RGST} and in the monograph~\cite{itenberg}. 


%

\section{Assumptions and preliminary results}\label{sec-prelim}

Consider a countable set $X$ (that is a finite or denumerable set),
endowed with a distance $d$,
such that bounded sets are finite. For instance one can consider
the set of natural numbers $\N$ or of integer numbers $\Z$,
with the distance $d(x,y)=|x-y|$, or the set $\Z^k$ for some $k$,
with the distance $d(x,y)=\|x-y\|$ where $\|\cdot\|$ is any norm on $\R^k$.
The previous property of the distance $d$ implies
that it defines the discrete topology
on $X$, that is all subsets of $X$ are open.
 In particular, the sets of finite, compact, and bounded subsets
of $X$ coincide. We shall denote them by $\KK$.

If $(s_K)_{K\in\KK}$ is a net with values in the set $\rbar$ of 
extended real numbers, indexed
by the compact sets of $X$, we use the notation:
\begin{align*}
\liminf_{K\in \KK } s_K:= \sup_{K\in \KK} \inf_{K'\in \KK,\; K'\supset K} s_{K'} \enspace.
\end{align*}
We define similarly $\limsup_{K\in \KK} s_K$ and if both quantities coincide 
we denote them by $\lim_{K\in \KK } s_K$, which we call the limit of
$s_K$ as $K$ tends to $X$.

Given a  \NEW{kernel}  on $X$,
$b:X\times X\to\rmax, \; (x,y)\mapsto b_{xy}$, 
which may be thought of as the square countable \NEW{matrix}
$B=(b_{xy})_{x,y\in X}\in \rmax^{X\times X}$,
a possible generalisation of the optimal assignment problem from the
finite to the countable state space case would be to consider the problem
\begin{equation}\label{assign-pb}
\text{find a bijection } F: X\to X \text{ maximising }
\limsup_{K\in \KK } \sum_{x\in K}  b_{xF(x)}\enspace ,\end{equation}
or the similar problem obtained by replacing the limsup in~\eqref{assign-pb}
by a liminf. As the limsup in~\eqref{assign-pb} may well be infinite, we 
shall rather use the following stronger definition:
\begin{definition}\label{defi-assign}
 A bijection $F : X \to X$
 is a (global) \NEW{solution}, resp.\  a  \NEW{strong solution},
of the \NEW{assignment problem}
associated to the kernel $b:X\times X\to\rmax$ if
 \begin{equation}
  \label{eq1.3}
\liminf_{K\in\KK } \sum_{x\in K} (b_{xF(x)}- b_{xG(x)})\ge 0\enspace ,
\end{equation}
resp.\ if 
\begin{equation}
  \label{eq1.4}
\liminf_{K\in \KK } \sum_{x\in K} (b_{xF(x)}- b_{xG(x)})> 0\enspace ,
\end{equation}
for any other bijection $G: X\to X$.
\end{definition}
If a strong solution exists, then it is obviously a unique 
solution to the assignment problem.

Given a kernel $b$, we define the \NEW{Moreau conjugacy} 
$B:\RBX\to\RBX$ which maps any
function $f=(f_x)_{x\in X}$ to the function $Bf=((Bf)_x)_{x\in X}$
such that
\begin{align} \label{eq1.1}
(Bf)_x=\sup_{y} (b_{xy} -f_y )
\end{align}
with the convention that $-\infty$ is absorbing for addition,
i.e., $-\infty+\lambda=\lambda+(-\infty)=-\infty$,
for all $\lambda\in\rbar$. Here, and in the sequel, the supremum is understood
over all the elements of $X$.
Like in~\cite{AGK2} and mainly for the sake of symmetry,
we work here with Moreau conjugacies~\eqref{eq1.1} rather than
with the max-plus linear maps discussed in the introduction.


We shall need the following assumptions on the kernel $b$:
\begin{myitemize}
\item[\ANZ]
 For any $x\in X $, there exist $y,z\in X$ such that $b_{xy} \neq -\infty$
and $b_{zx}\neq -\infty$.
\item[\ATC]
$\sup\set{b_{xy}}{ d(x,y)\geq n}$ tends to $-\infty$  when $n$ goes
 to infinity.
\end{myitemize}


Condition~\ANZ, which means that all the rows and columns of the matrix
$B$ are non zero (in the max-plus sense), was already used in~\cite{AGK2}.
Condition~\ATC\ is a tightness condition. It implies in particular that all
the rows and columns of $B$ are tight vectors or measures (related notions 
were defined  and used in~\cite{duality} for a  general topological space $X$,
and in~\cite{AGW} for a countable space $X$).
Under Condition \AC,  the Moreau conjugacy $B$
sends the set $\Bd(X)$ of real valued functions on $X$
that are bounded from below, to the set $\R^X$
of all real valued functions  on $X$.

By $B^T$ and $b^T$, we shall denote the transpose matrix of $B$
and its kernel, $B^T=(b^T_{xy})_{x,y\in X}$, $b^T_{xy}=b_{yx}$.
The corresponding Moreau conjugacy is then:
\begin{align*}
(B^Tg)_y =\sup_x (b_{xy}-g_x)\enspace .
\end{align*}

The pair $(B,B^T)$ defines a
Galois connection on $\RBX$, which means in particular (see
\cite{AGK2}) that $B^T$ is a pseudo-inverse of $B$ in the
sense that $B\circ B^T\circ B=B$ and $B^T\circ B\circ B^T=B^T$, hence
if the equation $Bf=g$ with a given $g\in \RBX$ has a
solution $f\in \RBX$, then necessarily $B^Tg$ is also a
solution of this equation. 

\bigskip

The infinite dimensional theory depends crucially on the 
class of functions in which the solutions to the equation $Bf=g$
are sought and on the class of bijections for the
solutions to the assignment problem. We first introduce some classes of
bijections.
\begin{definition} We define the \NEW{distance} between
two bijections $F,G : X \to X$ as 
$$
\rho (F,G)=\sup_{x}d(F(x),G(x))\in\R\cup\{+\infty\}\enspace.
$$
A bijection $F:X\to X$ is \NEW{locally bounded}
if it is at a finite distance from the identity map,
$I:X\to X,\; x\mapsto x$.
\end{definition}

The map $\rho$ satisfies all the properties of a distance except that
$\rho(F,G)$ may be infinite.
The binary relation defined as 
$F\sim G$ if $F$ and $G$ are at a finite distance 
($\rho (F,G)<\infty$) is clearly an
equivalence relation on the set of bijections of $X$, defining a
partition of this set into classes. 
The set of locally bounded bijections is the class of the
identity map.

\begin{pty}\label{pty-group}
The set of locally bounded bijections $X\to X$ is a subgroup of the group
of bijections of $X$.
\end{pty}
\begin{proof}
This follows from
$\rho(F^{-1},I)=\rho(I,F)$ and 
$\rho(F\circ G, I)\leq \rho(F\circ G, G)+\rho(G,I)=
\rho(F, I)+\rho(G,I)$
\end{proof}

\begin{definition} 
We say that a bijection $F:X\to X$ is a \NEW{local solution} 
(resp.\ a \NEW{local strong solution})
of the assignment problem associated to $b$
if Condition~\eqref{eq1.3} (resp.\ \eqref{eq1.4}) of
Definition~\ref{defi-assign}
holds for all $G$ within a finite distance from $F$.
\end{definition}

Now we define some classes of functions.
Recall that $\Bd(X)$ is the set of real valued functions on $X$,
$s=(s_x)_{x\in X}$ that are bounded from below, that is
$\inf_{x} s_x >-\infty$.
By $\ell_{\infty}=\ell_\infty (X)$, $\ell_1=\ell_1(X)$, 
$\ell_0=\ell_0(X)$, we shall denote the linear spaces (in the usual sense) 
of real valued functions on $X$, 
$s=(s_x)_{x\in X}$, such that respectively 
$\|s\|_\infty:=\sup_{x} |s_x|<\infty$, $\|s\|_1:=\sum_{x}
|s_x|<\infty$, or the limit $\lim_{x \to \infty} s_x$
exists and is finite. Here, the expression $x\to\infty$ 
refers to the filter of complements of finite sets
of $X$. Equivalently, we may choose arbitrarily a \NEW{basepoint}
$\bar x\in X$, and set $d(x):=d(x,\bar x)$. Then, $\lim_{x\to\infty} s_x = a$
if and only if $s_x$ tends to $a$ as $d(x)$ tends to infinity.
We shall also denote by $\ell_{0,1}=\ell_{0,1}(X)$ the linear space
of  functions $s=(s_x)_{x\in X}\in\ell_0$ such that for all $M>0$, 
$\|s\|_{0,1,M}:= \sup\set{\sum_{x} |s_{F(x)}-s_x|}{
F:X\to X,\text{ bijection s.t.}\; \rho(F,I)\leq M } <+\infty$.
This space can be thought of as the space of functions 
with $\ell_1$ ``partial derivatives'' and a limit at infinity.
In particular when $X=\Z^k$, all semi-norms $\|s\|_{0,1,M}$ are equivalent to
$\|s\|_{0,1}= \sum_{e\in E}\sum_{x} |s_{x+e}-s_x|$, where $E$ is the
canonical basis of $\R^k$.
For a general set $X$, in particular when $X$ is not included in 
a finite dimensional normed space (with the distance being defined 
from the norm), and when the cardinality of the
balls of radius $M$  in $X$ is not uniformly bounded,
one cannot find a finite set $E$ satisfying the above property, and one cannot 
replace the semi-norms $\|s\|_{0,1,M}$ by the following simpler ones 
$\|s\|'_{0,1,M}=\sum_{x} \max_{y,\, d(y,x)\leq M}|s_{y}-s_x|$.
Indeed, with these semi-norms, 
it may happen that $\ell_1\not\subset \ell_{0,1}$,
whereas the inclusion holds with  our definition of $\ell_{0,1}$,
as stated below.

\begin{pty}
We have $\ell_1(X)\subset\ell_{0,1}(X)\subset \ell_0(X)\subset\ell_\infty(X)\subset \Bd(X)$.\end{pty}
\begin{proof}
All these inclusions are clear, except perhaps the 
inclusion of $\ell_1$ in $\ell_{0,1}$ which follows from 
well known properties of series with positive terms:
if $s\in\ell_1$ then $s$ tends to $0$ at infinity and
$\sum_{x} |s_{F(x)}-s_x|\leq \|s\circ F\|_1+ \|s\|_{1} =2\|s\|_1$.
\end{proof}

By $\lb=\lb(X)$ we shall denote any of the former spaces.
They have the following good properties.

\begin{pty}\label{pty-lbinv}
For $\lb$ being either $\ell_1$, $\ell_0$ or $\ell_\infty$,
the space $\lb(X)$ is invariant by any bijection $F:X\to X$,
meaning that $\phi\circ F\in\lb(X)$ when $\phi\in \lb(X)$.
The space $\ell_{0,1}(X)$ is invariant by any locally bounded
bijection $X\to X$.
\end{pty}
\begin{proof}
This is clear for $\ell_\infty$. 
For $\ell_1$, this follows from properties
of series with positive terms.
For $\ell_0$, this follows
from the fact that, since the image by a bijection $F$
of any finite (compact) set of $X$ is finite,
the set $\KK$ of finite sets  is invariant by $F$: $F(\KK)=\KK$.
%
For $\ell_{0,1}$, let $\phi=(\phi_x)_{x\in X}\in\ell_{0,1}(X)$ and $F:X\to X$ be a
locally bounded bijection, and let us denote by $R=\rho(F,I)$.
Since $\ell_0$ is invariant by any bijection, then $\phi\circ F\in\ell_0$.
Now, for any $M>0$, 
and any bijection $G:X\to X$ such that $\rho(G,I)\leq M$, we have 
$\rho(F\circ G ,I)\leq M+R$, thus 
\begin{align*}
 \sum_{x} |(\phi\circ F)_{G (x)}-(\phi\circ F)_x|
&\leq \sum_{x} |\phi_{F\circ G(x)}-\phi_x |+ \sum_x |\phi_{F(x)}-\phi_x |\\
&\leq \|\phi\|_{0,1,M+R}+\|\phi\|_{0,1,R} \end{align*}
hence $\|\phi\circ F\|_{0,1,M}\leq \|\phi\|_{0,1,M+R}+\|\phi\|_{0,1,R}<+\infty$,
which shows that $\phi\circ F\in\ell_{0,1}$.
\end{proof}

We shall consider the following classes of solutions to the assignment problem.

\begin{definition} 
A bijection $F: X\to X$ is said to be a \NEW{$\lb$-bijection},
with respect to the kernel $b:X\times X\to\rmax$, if 
the sequence 
$(b_{xF(x)})_{x\in X}$ belongs to $\lb(X)$. When in addition $F$ is a solution
(in any sense) of the optimal assignment problem, we shall speak of $\lb$-solution. 
\end{definition}
\begin{remark}
In general, a solution of  the optimal assignment problem associated to 
the kernel $b$ is necessarily a solution of Problem~\eqref{assign-pb}, 
but the converse implication may not be true, because the supremum 
of the expressions in~\eqref{assign-pb} may be infinite.
However, if $F$ is a $\ell_1$-bijection, then it is a (strong) solution 
of the optimal assignment problem associated to the kernel $b$ if and only if
it is a (unique) solution of Problem~\eqref{assign-pb}.
\end{remark}

\begin{definition}\label{defi-strg-reg} 
A kernel $b$ 
or its corresponding Moreau conjugacy $B$ is said to be
\NEW{$\lb$-strongly regular} if there exists $g\in \lb$ such that (i) $f:=B^T g
 \in \lb$, (ii) $f$ is the unique solution $h$ in $\lb$ 
of the equation $Bh=g$ and (iii) $g$ is the unique solution $h$ in $\lb$
of the equation $B^Th=f$. In this case, $g$ (resp.\ $f$) is
 said to belong to the \NEW{$\lb$-simple image} of $B$ (resp.\ $B^T$).
\end{definition}

 Of course it follows from this definition that $B$ is $\lb$-strongly
 regular if and only if $B^T$ is $\lb$-strongly regular.

\begin{remark} One can show, see Remark~\ref{bij-sr-finite},
 that in the case of a finite set  $X$,
our definition  coincides with the standard definition of strong regularity
 given in the introduction and in~\cite{BH}.
 In fact we added Condition (iii) in our definition,
 which turns out to be automatically fulfilled for finite
sets $X$.
\end{remark}

\begin{definition} A matrix $B=(b_{xy})\in\rmax^{X\times X}$ (or
its kernel $b$) is \NEW{normal}
(resp.\ \NEW{strongly normal}) if all its non-diagonal  entries,
$b_{xy}$ with $x,y\in X$ and $x\neq y$, 
 are non-positive (resp.\ negative), and if all its diagonal entries,
$b_{xx}$ for  $x\in X$, are equal to $0$.
\end{definition}

This definition is literally the same as the usual finite-dimensional
one (see \cite{Bu}). The normal (resp.\ strongly
normal) matrices present a class of examples, where the identity map
is an obvious locally bounded $\lb$-solution (resp.\ strong
solution) to the assignment problem.
As our first result will show,
this class of matrices present natural ``normal forms'' for strongly
regular matrices.

\begin{definition} The kernels  $b,c:X\times X\to \rmax$ 
are 
\NEW{$\lb$-similar} if
 there exist two locally bounded bijections $H,K : X \to X$
  and two functions $\phi$ and $\psi$ from $\lb(X)$ such that
 \begin{equation}
  \label{eq1.5}
 c_{xy}= b_{H(x)K(y)}-\phi_x-\psi_y\enspace.
 \end{equation}
When $H$ (resp.\ $K$) is the identity map, we say that $b$ and $c$ are
{\em right} (resp.\ {\em left}) \NEW{$\lb$-similar}.
\end{definition}

When $X$ is finite, we recover the standard definition (see
 e.g.~\cite{Bu}). 
Indeed, matrices over the max-plus semiring are invertible if and only if 
they are the product of a permutation matrix and of a diagonal matrix with 
real diagonal entries.
So, similarity coincides with the usual notion that $C=PBP'$ 
for some invertible matrices $P$ and $P'$.

\begin{pty}
The relations of (right, left) $\lb$-similarity are equivalence relations.
\end{pty}
\begin{proof}
We first consider the relation of $\lb$-similarity. This relation
is reflexive since the identity map is locally bounded and the function $0$ (identically equal to $0$) is in $\lb$.

To see that it is symmetric, let $b$ and $c$  be $\lb$-similar, that is 
satisfying~\eqref{eq1.5} with locally bounded bijections $H$ and $K$, and
$\phi,\psi\in\lb(X)$. Then
 \begin{equation}
 b_{xy}= c_{H^{-1}(x)K^{-1}(y)}+(\phi\circ H^{-1})_x+(\psi\circ K^{-1})_y,
 \end{equation}
and by Properties~\ref{pty-lbinv} and~\ref{pty-group},
 $H^{-1}$ and $K^{-1}$ 
are locally bounded, and $\phi\circ H^{-1}$ and $\psi\circ K^{-1}$ are in
$\lb(X)$, which shows that $c$ and $b$ are $\lb$-similar.

Let us show that $\lb$-similarity is transitive. Assume that $b$ and $c$ are
$\lb$-similar and that $c$ and $c'$ are also $\lb$-similar.
This means that there exist locally bounded bijections $H,K,H',K'$
and functions  $\phi,\psi,\phi',\psi'\in\lb(X)$ satisfying~\eqref{eq1.5}
and
$
 c'_{xy}= c_{H'(x)K'(y)}-\phi'_x-\psi'_y$. 
Hence
$
 c'_{xy}= b_{H\circ H'(x)K\circ K'(y)}-(\phi\circ H'+\phi')_x-
(\psi\circ K'+\psi')_y$, 
and by Properties~\ref{pty-lbinv} and~\ref{pty-group}, and the linearity of 
$\lb(X)$, we get that $b$ and $c'$ are $\lb$-similar.

The relations of right and left $\lb$-similarity are treated
by requiring $H,H'$ or $K,K'$ to be the identity maps in the previous
arguments.
\end{proof}
\begin{remark}\label{rk-dimfinite}
In the finite dimensional case, linear programming (or network flow algorithms)
yields an effective method to reduce a matrix to a normal matrix
by similarity.
Indeed, the optimal assignment problem over a finite state space
can be formulated as a linear program, the dual of which can be written
as 
\[
\min_{\phi,\psi} \sum_x \phi_x +\psi_x,\qquad \phi,
\psi\in \R^X,\;\; \phi_x+\psi_y\geq b_{xy},\forall x,y  \enspace .
\]
The dual program has an optimal solution $(\phi^*,\psi^*)$,
except in the degenerate case in which the primal is not feasible
(meaning that there is no permutation
$F$ such that $b_{xF(x)}>-\infty$ for all $x$).
By complementary slackness, a permutation $F$ is optimal if and only if
the equality $\phi^*_x+\psi^*_y= b_{xy}$ holds whenever
$y=F(x)$. It follows that
the matrix $b_{xF(y)}-\phi_x-\psi_{F(y)}$, which is similar to $b$,
is normal.
\end{remark}

 The importance of the notion of $\lb$-similarity 
 is basically due to the following  results, which are countable
analogues to Propositions 3 and 4 in~\cite{BH}.

\begin{proposition}\label{prop1.1.i}
Conditions \ANZ, \ATC\ and $\lb$-strong regularity
are each invariant under $\lb$-similarity.
 \end{proposition}
\begin{proof} Let $b$ and $c$ be $\lb$-similar kernels on $X$,
thus satisfying~\eqref{eq1.5}
with locally bounded bijections $H,K:X\to X$ and $\phi,\psi\in\lb(X)$.

Since $\lb(X)\subset\R^X$, $b$ satisfies~\ANZ\ if and 
only if $c$ does. Moreover, since $H$ and $K$ are locally bounded,
we get for all $x,y\in X$:
\[ d(x,y)-\rho(H,I)-\rho(K,I)\leq d(H(x),K(y))\leq d(x,y)+\rho(H,I)+\rho(K,I)
\enspace,\]
hence $d(x,y)\to\infty$ if and only if  $d(H(x),K(y))\to\infty$, and 
since $\phi,\psi\in\lb\subset \ell_\infty$, we deduce that
$b$ satisfies~\ATC\ if and only if $c$ does.

The invariance of $\lb$-strong regularity follows from the
observation that 
$$ g=Cf \Leftrightarrow 
(g+\phi)\circ H^{-1}=B((f+\psi)\circ K^{-1} )\enspace,
$$
and so
$
 g=Bf \Leftrightarrow  
g\circ H-\phi= C(f\circ K-\psi)$. 
Indeed, let $b,f,g$ satisfy the properties of Definition~\ref{defi-strg-reg}.
Then, $g'=g\circ H-\phi\in\lb$, and since 
$
 f'=C^T g'\Leftrightarrow (f'+\psi)\circ K^{-1}= B^T((g'+\phi)\circ H^{-1})
$, 
and the last term in the previous equation is equal to $B^T g=f$,
we get that $f'=f\circ K-\psi\in\lb$, which shows Property~(i)
of Definition~\ref{defi-strg-reg} for $c,f',g'$ instead of $b,f,g$.
Moreover, we have $Ch'=g'$ if and only if $Bh=g$ for
$h=(h'+\psi)\circ K^{-1}$, and since $h\in\lb$ if and only if $h'\in\lb$,
we get that Property~(ii) of Definition~\ref{defi-strg-reg} for $c,f',g'$
is equivalent to the same property for $b,f,g$.
By symmetry, the same occurs for Property~(iii) of 
Definition~\ref{defi-strg-reg}.
\end{proof}

\begin{proposition}\label{prop1.1.ii}
The property for a kernel to have a solution
 or a local solution to the assignment problem
is invariant under $\ell_1$-similarity. The same
is true if the solution is required in addition
to be locally bounded, strong, or either a $\ell_1$, $\ell_0$ or
$\ell_\infty$-bijection, or a locally bounded $\ell_{0,1}$-bijection,
with respect to the kernel.
\end{proposition}
\begin{proof}
Let $b$ and $c$ be $\ell_1$-similar kernels on $X$,
thus satisfying~\eqref{eq1.5}
with locally bounded bijections $H,K:X\to X$ and $\phi,\psi\in\ell_1(X)$.
Let $F,G:X\to X$ be two bijections. We have for any $K\in\KK$,
\begin{eqnarray}
\sum_{x\in K} (c_{xF(x)}- c_{xG(x)}) &=& \sum_{y\in H(K)} (b_{y K\circ F\circ H^{-1}(y)}- b_{y K\circ G\circ H^{-1}(y)})  \label{eq1.6} \\
& &\quad 
+ \sum_{x\in K} ((\psi\circ G)_x- (\psi\circ F)_x)\enspace.  \nonumber
\end{eqnarray}
Since $\psi\in\ell_1$, the limit $\lim_{K\in\KK} \sum_{x\in K} \psi_x$ exists.
Moreover, since $F(\KK)=\KK$, we get
$
 \lim_{K\in\KK} \sum_{x\in K}  (\psi\circ F)_x=
\lim_{K\in\KK} \sum_{x\in F(K)}  \psi_x= \lim_{K\in\KK} \sum_{x\in K}  \psi_x
$, 
which implies that
$\lim_{K\in\KK} \sum_{x\in K} ((\psi\circ G)_x- (\psi\circ F)_x)  =0$.
Using this and $H(\KK)=\KK$ in~\eqref{eq1.6}, we deduce
\begin{align*}
 \liminf_{K\in\KK } \sum_{x\in K} (c_{xF(x)}- c_{xG(x)})
 = \liminf_{K\in \KK} \sum_{y\in K} (b_{y K\circ F\circ H^{-1}(y)}- b_{y K\circ G\circ H^{-1}(y)})\enspace.
\end{align*}
Since the map $\mathcal{T}(G):= K\circ G\circ H^{-1}$ is a bijective
transformation from the set of bijections $X\to X$ to itself,
we deduce from the latter relation 
that 
$F$ is a solution (resp.\ a strong solution) of the assignment problem associated to the kernel $c$ if and only if $K\circ F\circ H^{-1}$ is a solution
(resp.\ a strong solution) of the assignment problem associated to the kernel 
$b$.
Since $K$ is locally bounded, the map $\mathcal{T}$ is such that $G\sim G'\implies \mathcal{T}(G)\sim \mathcal{T}(G')$ (recall that $G\sim G'$ iff $\rho(G,G')<\infty$).
Since $H$ is also locally bounded, $\mathcal{T}$ is a bijective transformation from the set of locally bounded bijections to itself.
Hence, a solution (or strong solution, etc.)
$F$ for $c$ is locally bounded if and only if the corresponding solution
$K\circ F\circ H^{-1}$ for $b$ is locally bounded.
Moreover, we also deduce that $F$ is a local solution (resp.\ a local strong solution) of the assignment problem associated to the kernel $c$ if and only if $K\circ F\circ H^{-1}$ is a local solution (resp.\ a local strong solution) of the assignment problem associated to the kernel $b$.

Finally, assume that $F$ is a $\lb$-solution for some space $\lb$ (which may be different from $\ell_1$), that is $(c_{x F(x)})_{x\in X}\in\lb(X)$.
Composing this sequence with $H^{-1}$, we get that
$(c_{H^{-1}(x) F\circ H^{-1}(x)})_{x\in X} \in\lb(X)$.
Now by~\eqref{eq1.5}, we get that
$
 b_{x K\circ F\circ H^{-1}(x)}=
c_{H^{-1}(x) F\circ H^{-1}(x)}+(\phi\circ H^{-1})_x+(\psi\circ F\circ H^{-1})_x
$ 
and since $\phi,\psi\in \ell_1(X)\subset \lb(X)$,
we deduce that $(b_{x K\circ F\circ H^{-1}(x)})_{x\in X}\in\lb(X)$
if $\lb$ is either $\ell_1$, $\ell_0$ or $\ell_\infty$.
By symmetry, we have shown, in this case, that 
$F$ is a $\lb$-solution of the assignment problem associated
 to the kernel $c$ if and only if $K\circ F\circ H^{-1}$ is a
$\lb$ solution of the assignment problem associated  to the kernel $b$.
When $\lb$ is $\ell_{0,1}$,
we need to restrict solutions to be locally bounded.
\end{proof}

\begin{proposition}\label{prop1.1.iii}
The property for a kernel to have a local solution  to
the assignment problem is invariant under $\ell_{0,1}$-similarity.
The same is true if the solution is required in addition to be
locally bounded, strong, or either a $\ell_0$ or $\ell_\infty$-bijection,
or a locally bounded $\ell_{0,1}$-bijection, with respect to the kernel.
\end{proposition}
\begin{proof}
In view of the arguments of the proof of Proposition~\ref{prop1.1.ii}
it is enough to show  that 
\begin{equation}\label{def-sK}
 s_K= \sum_{x\in K} ((\psi\circ G)_{x}-(\psi\circ F)_x) \end{equation}
has a zero limit, $\lim_{K\in\KK} s_K=0$,
whenever $F$ and $G$ are bijections $X\to X$ that are at a finite distance
from each other, and $\psi\in\ell_{0,1}$. 
Since $\ell_{0,1}\subset \ell_0$,
any constant function is in $\ell_{0,1}$, and $s_K$ is invariant
when adding a constant to $\psi$, it suffices to consider
the case of functions $\psi\in\ell_{0,1}$ 
such that $\lim_{x\to\infty}\psi_x=0$.
Moreover, since 
$
s_K= \sum_{x\in F(K)} ((\psi\circ G\circ F^{-1})_{x}-\psi_x)$, 
$F(\KK)=\KK$ and $\rho(G\circ F^{-1},I)\leq \rho(G,F)<+\infty$,
we may assume that $F=I$ and that $G$ is locally bounded.

Let $M=\rho(G,I)<+\infty$, we get that
\begin{equation}\label{bkconv}
 \sum_{x\in X}|(\psi\circ G)_{x}-\psi_x|\leq  \|\psi\|_{0,1,M}<+\infty
\end{equation}
since $\psi\in\ell_{0,1}$.
Hence, the sequence $((\psi\circ G)_{x}-\psi_x)_{x\in X}$ is in 
$\ell_1$ which implies that $s_K$ is bounded, and,
by properties of series with positive terms, we get that
\begin{equation}\label{bkconv2}
 \limsup_{K\in \KK}\sum_{x\not\in K}|(\psi\circ G)_{x}-\psi_x|
= \inf_{K\in\KK } \sum_{x\not\in K}|(\psi\circ G)_{x}-\psi_x|=0\enspace .
\end{equation}
Hence $s_K$ has a limit.
Indeed, for any finite subsets $K_1$ and $K_2$ of $X$, we have 
\[
|s_{K_1}-s_{K_2}|\leq |s_{K_1}-s_{K_1\cap K_2}|+|s_{K_2}-s_{K_1\cap K_2}|
\leq 2  \sum_{x\not\in (K_1\cap K_2)}|(\psi\circ G)_{x}-\psi_x|\enspace,
\]
which implies that 
\begin{align*}
0\leq \limsup_{K\in\KK} s_K-\liminf_{K\in\KK} s_K
&=\inf_{K_1, K_2\in\KK} \sup_{K'_1\supset K_1,\; K'_2\supset K_2}
s_{K'_1}-s_{K'_2}\\
&\leq 2 \inf_{K_1, K_2\in\KK} \sum_{x\not\in (K_1\cap K_2)}|(\psi\circ G)_{x}-\psi_x|=0\enspace.
\end{align*}
To show that $s_K$ has a zero limit, it is thus sufficient to prove that
$\liminf_{K\in\KK} |s_K|=0$. Since this property means that
for all finite sets $K$, $\inf_{K'\supset K} |s_{K'}|=0$,
it will hold as soon as for any finite set $K$, there exists 
a sequence of finite sets $(K_n)_{n\geq 0}$ containing $K$ such that
$\lim_{n\to \infty}s_{K_n}=0$. 

Let us show this last property. 
Consider the  sequence $K_n$ such that
$K_0=K$ and $K_{n+1}=K_n\cup G(K_n)\cup G^{-1}(K_n)$ for $n\geq 0$.
Then $K_n$ is nondecreasing, and it satisfies $K_n\supset K$,
$G(K_n)\subset K_{n+1}$ and $G^{-1}(K_n)\subset K_{n+1}$.
We have 
\begin{align*}
 s_{K_n}&=\sum_{x\in G(K_n)} \psi_{x}-\sum_{x\in K_n} \psi_x
=\sum_{x\in G(K_n)\setminus K_n} \psi_x -\sum_{x\in K_n\setminus G(K_n)} \psi_x\\
&= \sum_{x\in G(K_n)\setminus K_n} \psi_x -\sum_{x\in G^{-1}(K_n)\setminus K_n} \psi_{G(x)}\enspace .\end{align*}
Since $G(K_n)\setminus K_n\subset K_{n+1}\setminus K_n$,
$G(K_n)\setminus K_n\subset G^{n+1}(K)$, hence 
its cardinality is less or equal to the cardinality $\#K$ of $K$, and 
the same is true for $G^{-1}(K_n)\setminus K_n$, we obtain
\begin{equation}\label{bound-bKn}
|s_{K_n}| \leq \#K \left(\max_{x\in K_{n+1}\setminus K_n} |\psi_x| + 
 \max_{x\in K_{n+1}\setminus K_n} |(\psi\circ G)_{x}|\right)\enspace. \end{equation}
Now the sets $K_{n+1}\setminus K_n$ are disjoint.
If $K_{n+1}\setminus K_n=\emptyset$ for some $n\geq 0$, then
$K_{n+1}=K_n$ and by construction $K_{m}=K_n$, hence 
$K_{m+1}\setminus K_m=\emptyset$ for all $m\geq n$.
This implies that $|s_{K_n}|=0$ for all $n\geq m$, hence the sequence
$(s_{K_n})_{n\geq 0}$  converges trivially to $0$.
Otherwise, if all the sets $K_{n+1}\setminus K_n$ are nonempty,
one can show, using the fact that they are all disjoint, 
that for all finite sets $K'$, 
$K_{n+1}\setminus K_n\subset X\setminus K'$ for $n$ large enough.
Since $\lim_{x\to\infty} \psi_x=0$, we deduce that 
$\max_{x\in K_{n+1}\setminus K_n} |\psi_x|$ tends to $0$.
Since the same is true for  $\psi\circ G$ instead of $\psi$, 
Inequality~\eqref{bound-bKn} implies that  the sequence
$(s_{K_n})_{n\geq 0}$ converges to $0$.
This concludes the proof.
\end{proof}

From the previous proof, it seems that with Definition~\ref{defi-assign}
of a solution to the assignment problem, the invariance by similarities fails 
under weaker assumptions on similarities,
in particular for $\ell_0$ and $\ell_\infty$-similarities.
This may hold however if we weaken also the definition of a solution
to the assignment problem as follows.
In the sequel, we fix a base point $\bar{x}$ and denote by
$B_n$ the ball of centre $\bar{x}$ and radius $n$ in $X$.

\begin{definition}\label{defi-assign-bis}
 A bijection $F : X \to X$
 is a (global) \NEW{restricted solution}, resp.\  a  \NEW{strong restricted
solution}, of the \NEW{assignment problem}
associated to the kernel $b:X\times X\to\rmax$ if
 \begin{equation}
\label{eq1.3bis}
\liminf_{n\to \infty} \sum_{x\in B_n} (b_{xF(x)}- b_{xG(x)})\ge 0\enspace ,
\end{equation}
resp.\ if 
\begin{equation}
  \label{eq1.4bis}
\liminf_{n\to \infty} \sum_{x\in B_n} (b_{xF(x)}- b_{xG(x)})> 0\enspace ,
\end{equation}
for any other bijection $G: X\to X$. We say that 
$F$ is a \NEW{local} (resp.\ \NEW{local strong}) \NEW{restricted solution},
if \eqref{eq1.3bis} (resp.\ \eqref{eq1.4bis}) holds for all 
$G$ within a finite distance from $F$.
\end{definition}

With this definition, we cannot change the ``order'' of rows of a matrix,
that is we need to consider right-similarities only.
From the same arguments as in the proofs of Propositions~\ref{prop1.1.ii}
and~\ref{prop1.1.iii}, we get that
\begin{proposition}\label{prop1.1.iv}
The conclusions of Propositions~\ref{prop1.1.ii} and~\ref{prop1.1.iii}
hold true if we replace ``solutions'' by ``restricted solutions'' and ``similarities'' by ``right-similarities'' in their statements.
\end{proposition}

Moreover, we can consider $\ell_0$-right-similarities.

\begin{proposition}\label{prop1.1.v}
Assume that $\# B_n-\# B_{n-1}$ is bounded.
Then, the property for a kernel to have a locally bounded
local restricted solution to
the assignment problem is invariant under $\ell_{0}$-right-similarity.
The same is true if the solution is required in addition to be
strong, or either a $\ell_0$ or $\ell_\infty$-bijection,
with respect to the kernel.
\end{proposition}
\begin{proof}
In view of the arguments of the proof of Propositions~\ref{prop1.1.ii}
and~\ref{prop1.1.iii},
it is enough to show  that $s_{B_n}$, defined by~\eqref{def-sK},
converges to $0$ when $n$ goes
to infinity, whenever $F$ and $G$ are locally bounded bijections
$X\to X$, and $\psi\in\ell_{0}$ has a zero limit. 
Moreover, taking the difference of $\psi_{G(x)}$ and $\psi_{F(x)}$ 
with $\psi_x$ in the expression of $s_{B_n}$, we may assume that $F=I$.
Then by the same arguments as in the proof of  Proposition~\ref{prop1.1.iii}
we get that 
$
 s_{B_n}=\sum_{x\in G(B_n)\setminus B_n} \psi_x -\sum_{x\in G^{-1}(B_n)\setminus B_n} \psi_{G(x)}$. 
Since $R:=\rho(G,I)<+\infty$, we get that
$G(B_n)\subset B_{n+R}$ and $G^{-1}(B_n)\subset B_{n+R}$, and 
by the assumption on the cardinality of $B_n$, 
we get that the cardinality of $G(B_n)\setminus B_n$ is bounded 
by some constant $M$.
Hence
\begin{equation}\label{bound-bKnbis}
|s_{B_n}| \leq M \left(\max_{x\in B_{n+R}\setminus B_n} |\psi_x| + 
 \max_{x\in B_{n+R}\setminus B_n} |(\psi\circ G)_{x}|\right)\enspace. \end{equation}
Since $\psi_x$ and $\psi_{G(x)}$ tend to $0$ when $x\to\infty$,
the r.h.s.\ of~\eqref{bound-bKnbis} tends to $0$, which 
implies that the sequence $(s_{B_n})_{n\geq 0}$ converges to $0$.
This concludes the proof.
\end{proof}

\section{Main results}\label{sec-form}
In this section, we state the main results, which we prove in Sections~\ref{sec2} and~\ref{sec3}.
\begin{theorem}\label{th1} A kernel satisfying \AC\ is
$\lb$-strongly regular if and only if it is $\lb$-similar to a
strongly normal kernel or if and only if it is $\lb$-right 
(resp.\ left)-similar to a strongly normal kernel.
\end{theorem}

The following counter-example shows that the tightness condition~\ATC\ is useful in the previous result.

\begin{ex}
Consider $X=\N$ and $b_{xy}=-1/|x-y|$ for $x\neq y$ and $b_{xx}=0$.
The kernel $b$ is clearly strongly normal. It satisfies Condition~\ANZ, but not
Condition~\ATC.
Let $f,g\in \ell_1(X)$ be such that $Bf=g$ and $B^Tg=f$.
We get that $g_x\geq \lim_{y\to\infty} b_{xy}-f_y=0$ and
symmetrically $f_y\geq 0$.
This implies that $g_x=\sup_{y} b_{xy}-f_y\leq 0$ and 
$f_y=\sup_{x} b_{xy}-g_x\leq 0$. Hence $f=g\equiv 0$.
However, the function $h\in\R^X$ such that
$h_x=1/(x+1)^2$ satisfies $Bh=g$ and $h\in \ell_1$ but $h\neq f$.
Hence $b$  is not $\ell_1$-strongly regular, 
thus, by Proposition~\ref{prop1.1.i},
it cannot be $\ell_1$-similar to a $\ell_1$-strongly regular kernel.
\end{ex}

Theorem~\ref{th1} shows in particular that a $\lb$-strongly regular kernel 
satisfying \AC\ is $\lb$-similar to a kernel having a strong solution to 
the assignment problem.
But it is of course interesting to know what can be said about the
assignment problem for the regular kernel itself.
In the analysis of this question (as well as
the inverse one),  an important role is played by the following
construction.

%

If $c: (x,y)\in X\times X \mapsto c_{xy}\in \rmax$ is a kernel,
we define the kernel $c^+: (x,y)\in X\times X \mapsto c^+_{xy}\in \rbar$,
\begin{align}\label{eq-def-cplus}
c^+_{xy}= \sup_{x_0,x_1,\ldots,x_n}c_{x_0x_1}+\cdots+c_{x_{n-1}x_n} \enspace, 
\end{align}
where the sup is taken over $n\geq 1$ and over all the sequences $x_0,x_1,\ldots,x_n$ of
elements of $X$ such that $x_0=x$ and $x_n=y$. 
The sum
$c_{x_0x_1}+\cdots+c_{x_{n-1}x_n}$ is the \NEW{weight} of the sequence
$x_0,\ldots,x_n$, so that $c^+_{xy}$ represents the maximal weight
of a path of positive length from $x$ to $y$. 

The sequence
$x_0,\ldots,x_n$ is said to
be a \NEW{circuit} if $x_0=x_n$. If every circuit has a nonpositive weight,
the supremum in~\eqref{eq-def-cplus} does not change if one restricts it to
those sequences such that the elements $x_1,\ldots,x_{n-1}$ are
pairwise distinct and are distinct from $x_0$ and $x_n$. Note however
that unlike in the case in which $X$ is finite, the fact that
every circuit has a nonpositive weight does not 
imply that $c^+_{xy}<\infty$ for all $x,y\in X$,
although this turns out to be automatically the case when
$c$ is \NEW{irreducible}, meaning that $c^+_{xy}>-\infty$
for all $x,y\in X$, see~\cite{AGW} for more details. 

It follows readily from the definition that 
$c^+_{xy}\geq c^+_{xz}+c^+_{zy}$. 
Let us now consider the vector $f_x:=c^+_{xy}$, for some arbitrary $y\in X$.
We deduce from the previous inequality that
$
f_x \geq \sup_{z} (c_{xz}+f_z) $, 
Moreover, when $c^+_{yy}\geq 0$, and a fortiori when $c_{yy}\geq 0$, it can be checked that the equality holds,
for all $x\in X$ (see for instance~\cite{AGW}). 

We shall now apply this construction to the kernel $c=\tilde{b}$
where
\begin{equation}\label{btilde}
\tilde{b}_{xy} = b_{xF(y)}-b_{yF(y)} \enspace,
\end{equation}
and  $F$ is a (possibly local) solution of the assignment problem 
associated to a kernel $b$.
The kernel $\tilde{b}^+$ is obtained by taking $c=\tilde{b}$ in
Equation~\eqref{eq-def-cplus}. Observe that $\tilde{b}_{xx}=0$ and
that the weight of any circuit, with respect to $\tilde{b}$, is
non positive.

As was observed in \cite{Ru}, the
functions $\tilde b_{xy}^+$ turn out to be useful also in the analysis
of the Monge-Kantorovich mass transfer problem, 
a natural analog of the assignment problem for general measurable,
(uncountable) state space $X$.

Define the \NEW{potential} and the \NEW{inverse potential} as the functions on $X$ given respectively by
 \begin{equation}
  \label{eq1.8}
\bar\phi_x=\sup_{y} \tilde b_{xy}^+\in\R\cup\{+\infty\},
 \qquad \bar\psi_y=\sup_{x} \tilde
b_{xy}^+\in\R\cup\{+\infty\} \enspace .
   \end{equation}
 The following simple
  properties of these functions are crucial:
  \begin{myenumerate}
\item  $\tilde b^+_{xx}$, $\bar\phi_x$ and $\bar \psi_y$ are nonnegative
 for all $x$ and $y$ (in fact, take $n=1$ in~\eqref{eq-def-cplus}).
\item  the function $f=\bar\phi$ satisfies the equation
 \begin{equation}
  \label{eq1.9}
  f_x=\sup_{y} (\tilde{b}_{xy}+f_y )\enspace, 
    \quad \forall x\in X.
   \end{equation}
Similarly, the function $g=\bar\psi$ satisfies the equation
 \begin{equation}
  \label{eq1.9bis}
  g_y=\sup_{x} (\tilde{b}_{xy}+g_x) \enspace. 
\end{equation}
Moreover,  if $\bar\psi$ (resp.\ $\bar\phi$) is finite, the function $-\bar\psi$ (resp.\ $-\bar\phi$) also satisfies~\eqref{eq1.9} (resp.~\eqref{eq1.9bis}).
Observe that Equation~\eqref{eq1.9} can be equivalently written as
\begin{equation}
  \label{eq1.10}
  f=B\psi, \qquad \psi_y=b_{F^{-1}(y)y}-f_{F^{-1}(y)}     \quad \forall y\in X
\enspace .
   \end{equation}
   \item  The function $f=\bar \phi$ and, if $\bar \psi$ is finite,
the function  $f=-\bar \psi$ satisfy the equation
 \begin{equation}
  \label{eq1.11}
  f_x=\sup_{y} (\tilde b_{xy}^++f_y)\quad \forall x\in X\enspace.
   \end{equation}
\end{myenumerate}

\begin{remark}\label{normal-potential}
When $b$ is a  normal kernel, taking $F$ to be the identity in 
the definition of the kernel $\tilde{b}$,
we get $\tilde b=b$,  $\tilde b^+_{xy}\leq 0$ and
$\bar\phi_x=\bar \psi_y=0$ for all $x,y$.
\end{remark}

\begin{theorem}\label{th2}
 {\rm (i)} If a kernel $b$ satisfying \AC\  is
 $\ell_{0,1}$-strongly regular, then it has a locally
bounded strong  local $\ell_{0,1}$-solution 
to its assignment problem.
Moreover, if $F$ denotes this (necessarily unique) solution,
and if $\tilde b$ is defined from $F$ by~\eqref{btilde},
the kernel $\tilde b^+$ satisfies: 
\begin{equation}
  \label{eq1.12}
\limsup_{x,y\to \infty} \tilde b^+_{xy}\leq 0
\end{equation}
and the potentials $\bar\phi$ and $\bar\psi$ (defined in~\eqref{eq1.8})
are bounded functions.

{\rm  (ii)} If
$b$ is $\ell_1$-strongly regular, then $F$ is also a global
strong $\ell_1$-solution to the assignment problem associated to $b$.

{\rm (iii)}  Under the assumption that $\# B_n-\# B_{n-1}$ is bounded,
if $b$ is $\ell_{0}$-strongly regular, then it has a locally
bounded strong local restricted $\ell_{0}$-solution $F$
to the assignment problem associated to $b$, and 
the  kernel $\tilde b^+$ and potentials $\bar\phi$ and $\bar\psi$ 
satisfy the properties of Point {\rm (i)}.
\end{theorem}

In order to prove a converse to Theorem~\ref{th2}, we shall need
the following additional technical assumption on a solution to the assignment
problem:

\def\BLC{{\rm (PC-$\lb$)}}
\begin{myitemize}
\item[\BLC] Either the potential $\bar\phi$ or the inverse potential
$\bar\psi$ associated to $b$ and $F$ belongs to $\lb(X)$.
\end{myitemize}

\begin{theorem}\label{th3} 
Let $b:X\times X\to\rmax$ be  a kernel satisfying \AC.
If $\lb$ is either $\ell_{0,1}$ or $\ell_1$, and 
if the assignment problem associated to $b$ has a (possibly local)
locally bounded strong $\lb$-solution $F$ satisfying Condition~\BLC,
then $b$ is $\lb$-strongly regular.
If $\lb$ is  $\ell_{0}$,  $\# B_n-\# B_{n-1}$ is bounded, and if 
 the assignment problem associated to $b$ has a local
locally bounded strong restricted $\lb$-solution $F$ satisfying Condition~\BLC,
then $b$ is $\lb$-strongly regular.
\end{theorem}

\begin{remark} We have to stress an unpleasant small gap between
necessary and sufficient conditions: from strong $\ell_{0,1}$-regularity it
follows that the potential $\bar\phi$ belongs to $\ell_{\infty}$, but in
Theorem~\ref{th3} we assume that $\bar\phi \in \ell_{0,1}$ (which implies
\eqref{eq1.12}).
 However, when considering classes of similar
kernels this discrepancy vanishes, as shown by the following direct
corollary of Theorem \ref{th1} and Remark~\ref{normal-potential}.
\end{remark}

\begin{corollary} 
A kernel $b$, satisfying \AC, is $\lb$-strongly regular if and only if it is
$\lb$-similar to a kernel having a strong solution to the
assignment problem satisfying condition \BLC.
\end{corollary}

In the case of a finite set $X$, the technical assumptions in
Theorems~\ref{th2} and \ref{th3} vanish, and we recover the result 
of Butkovi\v{c} and Hevery showing that
strong regularity is equivalent to the uniqueness of the
optimal assignment problem. This result was established 
in~\cite[Theorems 1 and 3]{BH} in which the authors considered more generally
matrices with entries in a dense commutative idempotent semiring.


\section{Coverings and subdifferentials. Proofs of Theorems~\ref{th1} and~\ref{th2}}\label{sec2}

 For the analysis of the equation $Bf=g$ (also in a more general
 setting of uncountable $X$) an important role belongs to the
 notion of generalised subdifferentials 
(see for instance~\cite{martinez88,martinez95,AGK2}).

\begin{definition}Let $b:X\times X\to \rmax$ be a kernel and $B$ its
associated Moreau conjugacy.
Given $f\in \RBX$ and $y\in X$,  the
\NEW{subdifferential} of $f$ at $y$ with respect $b$ or $B$,
denoted $\partial_b f(y)$ or $\partial f(y)$ for brevity is defined
 as 
 $$
 \partial f(y)=\set{x\in X} {
b_{xy}\neq-\infty,\quad (Bf)_x=\sup_z (b_{xz}-f_z)=b_{xy}-f_y }\enspace.
 $$
The subdifferential $\partial_{b^T}g(x)$ of $g\in\RBX$ at $x\in X$ with respect to $b^T$ will be denoted by $\partial^T g(x)$ for brevity:
$$   \partial ^T g(x)
   =\set{ y\in X}{
b_{xy}\neq-\infty , \quad
 (B^Tg)_y=\sup_z (b_{zy}-g_z)=b_{xy}-g_x}  \enspace . $$
\end{definition}
\begin{remark}
In the finite dimensional case, if $f,g$ are obtained from
optimal dual solutions of the optimal assignment problem (Remark~\ref{rk-dimfinite}), every optimal permutation is obtained by selecting precisely
one element $F(x)$ in each $\partial^T g(x)$ (in such a way that the same
element is never selected twice). A symmetrical interpretation holds
with $\partial f(y)$ and the inverse optimal permutation $F^{-1}$
\end{remark}
 For a given $f$ the subdifferential is a mapping
 from $X$ to the set $\cP(X)$ of
 subsets of $X$. For any such mapping $G$, the inverse mapping
 $G^{-1}: X \mapsto \cP(X)$ is defined as
  $G^{-1}(y):=\set{x}{y\in G(x)}$ for $y\in X$.
If $Y,Z\subset X$, we say that the family of subsets
$\{G(y)\}_{y\in Y}$ is a covering of $Z$ if $Z\subset \cup_{y\in Y} G(y)$.

We shall start with the following well known basic property of
subdifferentials that we prove here for the sake of completeness.

 \begin{proposition}\label{prop2.1}
 If $g=BB^Tg $, then $(\partial ^Tg)^{-1}=\partial B^Tg$.
 \end{proposition}

\begin{proof}
We have
$
(\partial ^T g)^{-1}(y)
   =\set{ x}{b_{x,y}\neq -\infty,\; (B^Tg)_y=\sup_z (b_{zy}-g_z)=b_{xy}-g_x }
$. 
The latter relation can be rewritten as $g_x=b_{xy}-(B^Tg)_y$, or equivalently
 $B(B^Tg)_x=b_{xy}-(B^Tg)_y$, which means that $x\in \partial
 (B^Tg)(y)$.
\end{proof}

When $X$ is finite, the following result is due to
Vorobyev~\cite{Vo}, see also Zimmermann~\cite[Chapter 3]{Zi.K}.
In~\cite[Theorem 3.5]{AGK2}, we proved a more general result 
which applies to the case of a general topological space $X$.

 \begin{proposition}\label{prop2.2}
Suppose that $b$ satisfies Conditions~\AC\ and that
$g\in \R^X$ is such that $B^Tg\in \Bd(X)$.
Then $B^Tg$ is a solution to the equation $Bf=g$ if and only
 if $\partial ^Tg(x) \neq \emptyset$ for all $x$ or equivalently
 if the family of the subsets $\{(\partial ^Tg)^{-1}(y)\}_{y\in X}$ is
 a covering of $X$.
 \end{proposition}

\begin{proof} This follows readily from 
Theorem 3.5 from~\cite{AGK2}. We only have to
observe that the assumption that $f=B^Tg\in \Bd(X)$
together with Condition~\ATC\ ensure that the set $\{y: b_{xy}-f_y\geq \beta\}$ is finite for
any $x\in X$ and $\beta \in \R$, which is the crucial condition
for the applicability of this theorem.
\end{proof}

\begin{definition} Let $G$ be a mapping from $X$ to the set of its
subsets $\cP(X)$ and let the family of subsets $\{G(y)\}_{y\in Y}$
be a covering of $Z$ with $Y,Z\subset X$.
 An element $y\in Y$ is called \NEW{essential}
(with respect to this covering) if
$
\cup_{z\in Y\setminus y} G(z)\not\supset Z$. 
The covering is called \NEW{minimal} if all elements of $Y$ are
essential.
\end{definition}


When $X$ is finite, the following result reduces to 
Vorobyev~\cite[Theorem  2.6]{Vo}, see also Zimmermann~\cite[Chapter 3]{Zi.K}.
In~\cite[Theorem 4.7]{AGK2}, we proved a more general result 
which applies to the
case of a general topological space $X$, but when $\EC=\RBX$ only.

 \begin{proposition}\label{prop2.3}
Assume that $b$ satisfies Conditions~\AC\ and that
$g\in \R^X$ is such that $B^Tg\in \EC$, where $\EC$ is 
a linear subspace of $\Bd(X)$ containing all the
maps  $\delta_y:X\to \R$ such that $\delta_y(x)=1$ 
if $x=y$ and $\delta_y(x)=0$ otherwise.

Then $B^Tg$ is the unique solution $f\in\EC$ of the equation $Bf=g$
if and only  if $\{(\partial ^Tg)^{-1}(y)\}_{y\in X}$ is
 a minimal covering of $X$.
 \end{proposition}
\begin{proof}
If $\EC$ were replaced by $\RBX$ in the statement of the proposition
while keeping the condition that $B^Tg\in\Bd(X)$, 
this would be a consequence of Theorem 4.7 from \cite{AGK2}.
This shows in particular the  ``if'' part of the proposition 
for all subspaces $\EC$.

Let us prove the ``only if'' by adapting the proof 
of~\cite[Theorem 4.7]{AGK2}.
Assume that $g\in \R^X$ is such 
that  $B^Tg\in \EC$ is the unique solution $f\in\EC$ of
the equation $Bf=g$.
By Proposition~\ref{prop2.2}, the family of subsets 
$\{(\partial ^Tg)^{-1}(y)\}_{y\in X}$ is  a covering of $X$.
Assume by contradiction that this covering is not minimal, 
i.e., that there exists $y_0\in X$ such that for all $x\in X$,
there exists $y\in X\setminus y_0$ such that
$x\in (\partial^T g)^{-1}(y)$.
This implies that $g_x=b_{xy}-(B^Tg)_y$, and since $g\geq  B B^T g$, we get:
\begin{equation}
g_x=\sup_{y\in X\setminus y_0} b_{xy} -(B^T g)_y \quad\forall x\in X\enspace.
\label{eq-n}
\end{equation}
Consider $f=B^Tg+\delta_{y_0}$. 
Since $B^Tg\in\EC\subset \R^X$ and $\delta_{y_0}\in
\EC$ and $\delta_{y_0}\not\equiv 0$, we obtain that $f\in \EC$ and 
that $f\neq B^T g$.
Since $f\geq B^T g$, we get that $Bf\leq B B^T g=g$.
Moreover, from~\eqref{eq-n},  we deduce the reverse inequality
$Bf\geq g$, hence $f$ is a solution of $Bf=g$, and we get a contradiction.
This concludes the proof.
\end{proof}

Proposition~\ref{prop2.3} can be applied in particular to $\EC=\Bd(X)$ or
to $\EC=\lb(X)$.

The key point in proving Theorems \ref{th1} and \ref{th2} is contained
in the following statement.

\begin{proposition}\label{prop2.4}
Let $\EC$ be as in Proposition~\ref{prop2.3}.
Suppose $g$ and $B^Tg$ belong to $\EC$ and are such that
$f=B^Tg$ is the unique solution $h\in\EC$ to the equation $Bh=g$
 and $g$ is the unique solution $h\in\EC$ to the equation $B^T h=f$.
Then there  exists a locally bounded bijection $F:X \to X$ such that
  \begin{equation} \label{eq2.1}
y=F(x) \iff \partial f(y)=\{x\}
 \iff \partial ^Tg (x)=\{y\} \enspace . 
\end{equation}
In particular
\begin{subequations}
\label{eq2.2}
\begin{gather}
b_{xF(x)}=g_x+f_{F(x)}\label{eq2.2-1}\\
\forall z\neq F(x) \quad  g_{x} > b_{xz}-f_z \enspace,
 \qquad
 \forall z\neq x \quad f_{F(x)} > b_{zF(x)}-g_z \enspace .\label{eq2.2-2}
 \end{gather}
\end{subequations}
 \end{proposition}

\begin{remark} As one easily checks, the inverse statement holds as
well: if a locally bounded bijection $F$ and if the functions
$f, g\in \EC$ satisfy \eqref{eq2.1}, then $f=B^Tg$ is the unique solution
$h\in\EC$ to equation $Bh=g$
 and $g$ is the unique solution $h\in\EC$ to the equation $B^Th=f$.
\end{remark}

\begin{proof}[Proof of Proposition~\ref{prop2.4}]
  Applying Proposition \ref{prop2.3} to the equation $B^Th=f$ one
concludes that for all $x$ there exists $y$ such that
 $y\in (\partial f)^{-1}(x)$, but $y\notin (\partial f)^{-1}(z)$
 for any $z\neq x$. In other words $(\partial f)(y)=\{x\}$,
 which by Proposition \ref{prop2.1} means that
 $(\partial^T g)^{-1}(y)=\{x\}$. Hence, defining the mapping
 $F : X\to \cP(X)$ by the formula
\begin{equation}\label{defF}
 F(x) =\set{ y}{(\partial f)(y)=\{x\}}
  =\set{ y }{(\partial^T g)^{-1}(y)=\{x\}}\subset \partial^T g (x)\enspace ,
\end{equation}
we deduce that $F(x) \neq \emptyset $ for all $x$ and
 $F$ is injective in the sense that $F(x)\cap
 F(z)=\emptyset$ whenever $x\neq z$. Applying now Proposition
 \ref{prop2.3} to the equation $Bf=g$ one finds that
 for all $y$ there exists $x$ such that
 $(\partial^T g)(x)=\{y\}$. From this one easily concludes that
 each set $F(x)$ contains precisely one point and that $F$ is
 surjective, which finally implies that $F$ is a bijection
 $X\to X$ such that \eqref{eq2.1} holds.

From the definition of $\partial f$ and $\partial^T g$, and
$Bf=g$, $B^Tg=f$, we deduce from~\eqref{eq2.1} that
$g_x=b_{xF(x)}-f_{F(x)}$ and that 
$g_{z}>b_{z F(x)}-f_{F(x)}$ for $z\neq x$,
and $f_{z}>b_{xz}-g_x$ for $z\neq F(x)$,
from which~\eqref{eq2.2} follows.

Let us show that $F$ is locally  bounded. 
Indeed,
 since $f$ and $g$ are bounded from below,
we get that  $b_{xF(x)}=f_{F(x)}+g_x$ is bounded from below,
but since $b$ satisfies \ATC, this implies that 
$d(x,F(x))$ is bounded.
\end{proof}

\begin{remark}\label{bij-sr-finite}
When $X$ is a finite set, the injectivity of the map $F$ defined 
in~\eqref{defF} implies automatically  that $F(x)$ contains exactly
one point and that it is a bijection.
Hence, in that case, the proof of Proposition~\ref{prop2.4}
only needs the assumption that $B^Th=f$ 
has a unique solution $h$, and the proof is thus much shorter.
By symmetry, in that case, one can also prove Proposition~\ref{prop2.4}
using the only assumption that $Bh=g$ has a unique solution $h$, which is
the definition of strong regularity given in~\cite{BH}.
From Proposition~\ref{prop2.3} (or from~\cite[Theorem  2.6]{Vo}),
one can also deduce that, when $X$ is finite,
the two assumptions are equivalent, and thus our definition 
of strong regularity is equivalent to that of~\cite{BH}, when the set
$X$ is finite.
\end{remark}

\begin{proof}[Proof of Theorem \ref{th1}] 
Let $b$ satisfies Condition~\AC.
If $b$ is $\lb$-similar to a strongly normal kernel $c$, then
by Proposition~\ref{prop1.1.i}, $c$ also satisfies~\AC.
Now taking for $g$ the zero function,
we get that $g\in\lb(X)$ and $f=0\in\lb(X)$.
Moreover, $\partial_c f(y)=\{y\}$ and $\partial_{c^T} g(x)=\{x\}$,
thus the covering of $X$ by
$\{ (\partial_{c^T}g)^{-1}(y)\}_{y\in X}$ is minimal, and by
Proposition~\ref{prop2.3}, the equation $Bh=g$ has a unique solution
$h\in\lb(X)$. Similarly, the equation $B^Th=f$ has a unique solution
$h\in\lb(X)$. This shows that $c$ is $\lb$-strongly regular.
Hence, by Proposition~\ref{prop1.1.i}, $b$ is also  $\lb$-strongly regular.
This shows the ``if'' part of the assertion of  Theorem \ref{th1}.

Let us show the ``only if'' part.
Assume now that $b$ is $\lb$-strongly regular, that is there exists 
$f,g\in\lb(X)$ such that $f=B^T g$, and the equations $Bh=g$ and 
$B^Th=f$ have both a unique solution in $\lb(X)$.
By Proposition~\ref{prop2.4}, there exists a locally bounded
bijection $F:X\to X$ satisfying~\eqref{eq2.2}.
From these equations, we deduce that the kernel $c:X\times X\to\rbar$ such that
$  c_{xy}= b_{xF(y)}-f_{F(y)}-g_x  $
is strongly normal.
Since $f\in\lb(X)$ and $F$ is locally bounded, $f\circ F\in\lb(X)$,
and since $g\in\lb(X)$, we deduce that $c$ is $\lb$-right-similar to $b$.
Similarly the kernel $c_{F^{-1}(x)F^{-1}(y)}$ is strongly normal and
 $\lb$-left-similar to $b$.
This finishes the proof of the theorem.
\end{proof}

\begin{proof}[Proof of Theorem~\ref{th2}] 
Let $b$ be a kernel satisfying \AC.
Assume that $b$ is $\lb$-strongly regular, and 
let $f$ and $g$ be as in Definition~\ref{defi-strg-reg}.
By  Proposition~\ref{prop2.4}, there exists a locally bounded
bijection $F:X\to X$ satisfying~\eqref{eq2.2}.
From these equations, we deduce that if $G:X\to X$ is another bijection,
$
   b_{xF(x)}-f_{F(x)} \geq b_{xG(x)}-f_{G(x)} $ 
for all $x\in X$, and that  the inequality is strict when $G(x)\neq F(x)$.
Hence if $G\neq F$, we get that
\begin{equation}\label{th34-eq1}
 \liminf_{K\in\KK } \sum_{x\in K} (b_{xF(x)}- b_{xG(x)})
> \liminf_{K\in\KK } \sum_{x\in K}(f_{F(x)}-f_{G(x)})\end{equation}
as soon as the r.h.s.\ of this inequality is finite. But,
the same arguments as in Propositions~\ref{prop1.1.ii} and~\ref{prop1.1.iii}
show that the r.h.s.\ of~\eqref{th34-eq1} is a limit and is equal to $0$
when either $f\in\ell_1$, or $f\in\ell_{0,1}$ while 
$F$ and $G$ are at a finite distance.
This shows that, when $\lb=\ell_1$, $F$ is a strong (global) solution
to the assignment problem associated to $b$, and that,
when $\lb=\ell_{0,1}$, $F$ is a strong local solution. 
Similarly,
\begin{equation}\label{th34-eq2}
 \liminf_{n\to\infty } \sum_{x\in B_n} (b_{xF(x)}- b_{xG(x)})
> \liminf_{n\to\infty } \sum_{x\in B_n}(f_{F(x)}-f_{G(x)})\end{equation}
as soon as the r.h.s.\ of this inequality is finite. But,
the same arguments as in Proposition~\ref{prop1.1.v}
show the r.h.s.\ of~\eqref{th34-eq2} is $0$ when $f\in\ell_0$, $F$ and $G$ are 
locally bounded and $\# B_n-\# B_{n-1}$ is bounded.
This shows that when $\# B_n-\# B_{n-1}$ is bounded, and
$\lb=\ell_0$, $F$ is a strong local restricted solution
to the assignment problem associated to $b$.

Since $b_{xF(x)}=g_x+f_{F(x)}$
and $\lb$ is invariant by any locally bounded bijection, we get that
$(b_{xF(x)})_{x\in X}$ is in $\lb$ and thus $F$ is a $\lb$-bijection.
Moreover, by the 
uniqueness of a strong local solution or of a strong restricted local solution,
the solutions $F$ obtained for the $\ell_1$ and $\ell_{0,1}$ cases are the same
under the assumptions of Point 
{\rm (ii)}, and  the solutions for the $\ell_{0,1}$
 and $\ell_{0}$ cases are the same 
under the assumptions of Point {\rm (i)} and the assumption that
$\# B_n-\# B_{n-1}$ is bounded.

It remains to show the properties of 
$\tilde b$ defined from $F$ by~\eqref{btilde},
and of the potentials $\bar\phi$ and $\bar\psi$ defined by~\eqref{eq1.8}.
From~\eqref{eq2.2}, we deduce that 
$\tilde b_{xy}\leq g_x-g_y$ for all $x,y\in X$, hence
\begin{equation}\label{bplus}
\tilde b^+_{xy} \le g_x-g_y.
\end{equation}
Since $g\in\lb\subset \ell_0$ for all cases of $\lb$ considered in
 Theorem~\ref{th2}, 
the r.h.s.\ of the above inequality~\eqref{bplus} tends to 0 
as $x,y \to \infty$, which shows~\eqref{eq1.12}.
Moreover, by definition of $\bar\phi$ and $\bar\psi$,
we get that 
$
\sup_{x} \bar\phi_x=\sup_{y} \bar\psi_y=\sup_{x,y} \tilde b^+_{xy} 
$ 
and using~\eqref{bplus} and the boundedness of $g$, we get that the
functions $\bar\phi$ and $\bar\psi$ are bounded from above.
Since they are also nonnegative functions, they are necessarily bounded.
\end{proof}

\section{"Perestroika" algorithm: proof of Theorem~\ref{th3}}\label{sec3}

Suppose the assumptions of Theorem \ref{th3} hold true
for one of the sets  $\lb$ considered in the statement.
Let $F$ be a
locally bounded strong local $\lb$-solution with $\lb$ being
either $\ell_{0,1}$ or $\ell_1$, or a locally bounded strong local restricted
$\ell_0$-solution to the assignment problem 
associated to the kernel $b$, satisfying condition \BLC. 
We shall consider the case where the potential $\bar\phi$ 
defined in~\eqref{eq1.8} belongs to $\lb(X)$ (the case with the
inverse potential is dealt with similarly). Since $\lb\subset\ell_0$,
this assumption together with Equation \eqref{eq1.11} implies 
Condition~\eqref{eq1.12}.

By Propositions~\ref{prop1.1.i}, \ref{prop1.1.ii}, \ref{prop1.1.iii}
and~\ref{prop1.1.v}
replacing the kernel $b$ with the $\lb$-right-similar kernel $c$ 
such that $c_{xy}=\tilde b_{xy}+\bar\phi_y-\bar\phi_x$, with $\tilde b$
as in~\eqref{btilde}, changes neither  Condition~\AC, nor
the property of $\lb$-strong 
regularity, nor the above property of having  a locally bounded strong local
$\lb$-solution (resp.\ restricted $\ell_0$-solution)
 to the assignment problem when $\lb$ is $\ell_1$ or $\ell_{0,1}$ (resp.\
$\ell_0$). Moreover, by the
proof of Proposition~\ref{prop1.1.ii}, 
we see that the solution of the assignment problem associated to the
kernel $c$ is the identity map. Since the diagonal entries of $c$ vanish,
we get that $\tilde c=c$, and by~\eqref{eq1.9} for $\bar\phi$ we get that
all the entries of $c$ are nonpositive, hence $c$ is a normal kernel.

Therefore, denoting the new  kernel again by $b$, 
we are reduced to the case where $b$ is a normal kernel and
$F$ is the identity map. From now
on, we shall suppose (without loss of generality) that these
additional simplifying conditions hold true. 
Hence $b$ satisfies the following conditions:
\begin{myitemize}
\item[\NCA]
$b$ is a normal kernel, satisfying Conditions \AC,
and the identity map of $X$ is a strong local 
solution or a strong restricted local solution
of its associated assignment problem.
\end{myitemize}
This implies in particular that the potential function $\bar\phi$ associated 
to $b$ is identically equal to $0$. 
In order to prove Theorem~\ref{th3}, we need 
to show that $b$ is necessarily $\lb$-strongly regular.
By Theorem~\ref{th1}, and the fact that $b$ satisfies Conditions~\AC,
it is enough to show that $b$ is $\lb$-right-similar to a strongly 
normal kernel.
To this end, we shall construct a function $\phi\in\ell_1(X)$ such that
\begin{equation}
\label{subeigen}
b_{xy}+\phi_y<\phi_x\end{equation}
for all $x\neq y\in X$, since then  the kernel $c$ with entries
$c_{xy}=b_{xy}+\phi_y-\phi_x$ would be strongly normal and
$\lb$-right-similar to $b$ ($\ell_1\subset \lb$).
Note that since $b$ satisfies Condition~\ATC\ and $\phi$ is bounded,
then~\eqref{subeigen} is equivalent to 
the condition:
$(A(-\phi))_x< \phi_x$, 
for all $x\in X$, or to the condition:
$(A^T\phi)_y< -\phi_y$, 
for all $y\in X$, where $A$ is the Moreau conjugacy associated
to the kernel $a$ which coincides with $b$ except on the diagonal
where it is equal to $-\infty$ ($a_{xy}=b_{xy}$ if $x\neq y$ and
$a_{xx}=-\infty$).

Given a function $\phi\in\Bd(X)$ and a kernel $b:X\times X\to\rmax$
satisfying
\begin{equation} \label{eq3.2}
b_{xy}+\phi_y\leq \phi_x \quad \forall x\neq y,
\end{equation}
we define the \NEW{saturation graph} associated to $\phi$ and $b$,
denoted by $\sat (b,\phi)$, or simply $\sat$ or $\sat(\phi)$,
as the (infinite) oriented graph whose edges
consist of the pairs $(x,y)\in X\times X$ such that $x\neq y$ and
\begin{equation} \label{eq3.1}
b_{xy}+\phi_y =\phi_x
\end{equation}
and whose set of vertices $\vert=\vert(b,\phi)$ is 
the subset of elements of $X$ that are adjacent to an edge. 
As usual by a \NEW{path} of length $n\geq 1$ in an oriented graph $G$
we mean a finite sequence $(x_1,\ldots,x_{n+1})$ of vertices such
that $(x_k,x_{k+1})$ is an edge for all $k=1,\ldots,n$.
and by a \NEW{circuit} (of length $n$)
we mean a path $(x_1,\ldots,x_{n+1})$ such that $x_{n+1}=x_1$.
An \NEW{infinite path}
leaving (resp.\ entering) the vertex $x$ of $G$ is a sequence $(x_n)_{n\geq 0}$ 
(resp. $(x_n)_{n\leq 0}$) such that $x_0=x$ and $(x_k,x_{k+1})$
is an edge for all $k\geq 0$ (resp.\ $k<0$).
A \NEW{string}  of $G$ is a sequence $(x_n)_{n\in\Z}$
such that $(x_n,x_{n+1})$ is an edge for all $n\in\Z$. 
The length of an infinite path or of a string is infinity.
The main properties of the saturation graph 
associated to the kernel $b$ are collected in the following statement.

 \begin{proposition}\label{prop3.1}
Let $b:X\times X\to\rmax$ be a kernel satisfying Condition~\NCA,
and $\phi\in\ell_1(X)$ satisfy~\eqref{eq3.2},
and denote by $\sat$ their saturation graph and by $\vert$ the set of
its  vertices. Then
{\rm (i)} $\sat$ contains no circuits nor strings. 
{\rm (ii)}~For all $x\in \vert$,
the set of edges  entering or leaving $x$ is finite.
{\rm (iii)}~For all $x\in \vert$, denote by $\larc(x)$  (resp.\  $\earc(x)$)
the supremum of the lengths of all the paths leaving $x$ (resp.\ entering $x$).
Then either $\larc(x)$ or $\earc(x)$ is finite.
{\rm (iv)}~If $\vert$ is nonempty, then the set of its end points
is nonempty, where by an end point we mean either an initial point
(no edge is entering it) or a final point (no edge is leaving it).
 \end{proposition}

\begin{proof}
Let us first note that since $b$ satisfies~\ATC, and $\phi$ is bounded,
there exists $M>0$ such that \eqref{subeigen} holds for all
$x,y$ such that $d(x,y)> M$.
This implies that all edges $(x,y)$ of $\sat$ satisfy $d(x,y)\leq M$.

\noindent (i) Suppose now that $\sat$ has a circuit $(x_1,\ldots,x_{n+1}=x_1)$. 
We can assume without loss of generality that this circuit is 
elementary, that is all vertices $x_k$ with $k=1,\ldots, n$ are distinct.
Hence, one can construct a bijection $G:X\to X$ 
which coincides with the identity map 
$F$ outside the elements of the circuit, and which acts as 
$x_k\mapsto x_{k+1}$ on the vertices of the circuit.
It is clear that $G$ is locally bounded and different from $F$,
and  since 
$$
 b_{x_1x_2}+\cdots+b_{x_{n-1}x_n}+b_{x_nx_1}=b_{x_1x_1}+\cdots+b_{x_n,x_n}\enspace,
$$
\eqref{eq1.4} does not hold,
 which contradicts the assumption that the identity map is a strong 
or a strong restricted local solution.

Assume next that $\sat$ contains a string $(x_n)_{n\in\Z}$.
Since $\sat$ contains no circuit, all elements $x_n$ of this sequence
are distinct. 
Hence one can construct a bijection $G:X\to X$
which coincides with the identity map 
$F$ outside the elements of the string, and which acts as 
the shift $x_k\mapsto x_{k+1}$ on the string. This bijection is
necessarily different from $F$. Moreover, 
since the distance between the vertices of an edge is bounded by $M$,
the bijection $G$ is locally bounded: $\rho(G,I)\leq M$.
Finally,  by \eqref{eq3.1}, we have 
$$
 \sum_{x\in K} (b_{xF(x)}-b_{xG(x)})=\sum_{x\in K} (\phi_{G(x)}-\phi_x)
 $$
and by the same arguments as in
the proof of Proposition~\ref{prop1.1.ii}, this sum has a zero limit.
This contradicts~\eqref{eq1.4} or~\eqref{eq1.4bis}, and thus 
the assumption that the identity map is a strong or a strong restricted 
local solution.

\noindent  (ii) Since all edges $(x,y)$ of $\sat$ satisfy $d(x,y)\leq M$,
we see that, for all $x\in \vert$,
the set of edges  entering or leaving $x$ is included in the ball of centre 
$x$ and radius $M$ which is finite.

\noindent  (iii) Choose $x\in \vert$. As there are no strings in $\sat$, either all paths
 leaving $x$ or all paths entering $x$ are finite. Consider, say, the
 first case. Suppose by contradiction that $\larc(x)=\infty$, that is
the lengths of the paths leaving $x$ are not  bounded. 
Hence,
\[ \infty=\larc(x)=\sup_{y} \larc(y)\enspace,\]
where the supremum is taken over the vertices $y$ such that $(x,y)$ is
an edge of $\sat$. By Point (ii), this set is finite, from which 
we deduce that at least one of its elements $y$ is such that $\larc(y)=\infty$.
Hence by induction one can construct an infinite path leaving $x$,
which contradicts our assumption.

\noindent  (iv) Again the absence of strings implies that each point belongs to a path that either ends
 in a final point or starts at an initial point.
\end{proof}

\begin{proof}[Proof of Theorem \ref{th3}]
Since $b$ is a normal kernel, the function $\phi\equiv 0$ 
satisfies~\eqref{eq3.2}, 
where the equality holds only on the edges $(x,y)$ of the graph $\sat(0)$.
Our goal is to change $\phi$ (by a successive ``perestroika'') in such a
way that no equality is left,
which would yield to~\eqref{subeigen} for all $x\neq y$.
We shall do this by successive elimination of the
end points of $\sat(0)$.

Namely, let $\phi\in \ell_1(X)$ satisfy~\eqref{eq3.2},
and let us denote respectively by $I_0=I_0(\phi)$ and $F_0=F_0(\phi)$ 
the sets of the initial points and final points of the saturation graph
$\sat(\phi)$.
By Point (iv) of
Proposition~\ref{prop3.1}, we know that either $I_0$ or $F_0$ is nonempty.
Assume for instance that $F_0$ is nonempty and let $x\in F_0$.
Then $b_{xz}+\phi_z<\phi_x$ for all $z\neq x$, 
and $b_{yx}+\phi_x=\phi_y$ for at least one vertex and at most a 
finite number of vertices $y\neq x$ of $\sat(\phi)$.
The first inequality implies that $(A(-\phi))_x<\phi_x$ (by Condition \ATC),
hence it is possible to decrease the value of $\phi$ in 
all final points without changing it elsewhere, 
in such a way that~\eqref{eq3.2} still holds for the new function $\phi'$,
and that $\sat(\phi')$ is equal to the subgraph of $\sat(\phi)$
where all final vertices and all edges entering them are removed.
In particular $V(\phi')=V(\phi)\setminus F_0(\phi)$.
Moreover, for any given function $\psi\in\ell_1(X)$ with positive values,
we can choose $\phi'$ in such a way that $|\phi'_x-\phi_x|\leq \psi_x$
for all $x\in X$, which will imply in particular that $\phi'\in\ell_1(X)$.
Indeed, let us take
$\phi'_x=\phi_x-\min(\psi_x, (\phi_x-(A(-\phi))_x)/2<\phi_x$ for all
$x\in F_0$ and $\phi'_x=\phi_x$ elsewhere.
Since $\phi'\leq \phi$, we get that 
$b_{yz}+\phi'_z\leq b_{yz}+\phi_z$ for all $y,z\in X$ such that
$z\neq y$, with equality if and only if $z\not \in F_0$.
Hence, $b_{yz}+\phi'_z\leq \phi'_y$ for all $y\in X\setminus F_0$ and
$z\neq y$, with equality if and only if
$(y,z)$ is an edge of $\sat(\phi)$ and  $z\not \in F_0$.
Moreover, for $y\in F_0$ and $z\neq y$, we have  
$b_{yz}+\phi_z\leq (A(-\phi))_y<\phi'_y$ , hence 
$b_{yz}+\phi'_z< \phi'_y$, and $(y,z)$ is not an edge of $\sat(\phi')$.


Let us now fix a function $\psi\in\ell_1(X)$, and denote by
$\pcaf(\phi)$ the function $\phi'$ obtained from $\phi$ by the previous
construction on the final points of $\sat(\phi)$.
We denote also by $\pcai(\phi)$ the function $\phi'$ 
obtained from $\phi$ by a similar construction where final points
 are replaced by initial points (or equivalently the kernel $b$ is replaced by
$b^T$ and the functions by their opposite).
This is one step of our ``perestroika'' algorithm.

Now, starting from any function $\phi^0\in \ell_1(X)$ satisfying~\eqref{eq3.2},
in particular the function $\phi^0=0$,
one can construct a sequence of functions
$\phi^n\in\ell_1(X)$ by $\phi^{n+1}=\pcaf(\phi^n)$.
At each step we have $V(\phi^{n+1})=V(\phi^n)\setminus F_0(\phi^n)$.
Hence, since $\phi^{n+1}-\phi^n$ has zero entries outside $F_0(\phi^n)$
and all these sets are disjoint, we get that for all $x\in X$,
$\phi^n_x$ converges in finite time towards some real $\phi_x$,
and since $|\phi^n-\phi^0|\leq \psi$ for all $n\geq 0$, 
the function $\phi=(\phi_x)_{x\in X}$ is in $\ell_1(X)$.
Note that the sequence $\phi^n$ may stop at step $n$ if $F_0(\phi^n)=\emptyset$,
in which case, $\phi$ will be simply this $\phi^n$.
Now, since $\phi^n_x$ converges in finite time for all $x\in X$,
we get easily that $\phi$ satisfies~\eqref{eq3.2} and
that $\sat(\phi)=\cap_{n\geq 0} \sat(\phi^n)$.
We can then start from $\psi^0=\phi$, and construct similarly a 
sequence $\psi^n$ using the algorithm $\pcai$ for initial sets.
The limit $\psi$ is again in $\ell_1(X)$,
satisfies~\eqref{eq3.2} and $\sat(\psi)=\cap_{n\geq 0} \sat(\psi^n)$.

Let us prove that $\sat(\psi)$ is empty or equivalently that
$V(\psi)=\emptyset$, in which case we would have
shown that $\psi$ satisfies~\eqref{subeigen} for all $x\neq y$.
For all $n\in\N\cup\{\infty\}$, we shall consider the following subsets
of the set of vertices of the saturation graph associated to
$\phi$:
\[ F_n(\phi):=\set{x\in V(\phi)}{\larc(x)=n},\quad
I_n(\phi):=\set{x\in X}{\earc(x)=n}\enspace.\]
By Point (iii) of Proposition~\ref{prop3.1}, we know that
for any $\phi\in \ell_1(X)$,
and $x\in V(\phi)$, either $\larc(x)$ or $\earc(x)$ is finite,
hence 
\[  V(\phi)=\bigcup_{n\in\N\cup\{\infty\}} F_n(\phi)=
\bigcup_{n\in\N\cup\{\infty\}} I_n(\phi) 
\quad \text{and}\quad  
F_\infty(\phi)\subset \bigcup_{n\in \N} I_n(\phi)\]
where the unions are disjoint.
But the ``perestroika'' algorithm for final points 
is such that $\sat(\pcaf(\phi))$ is equal to the subgraph of $\sat(\phi)$
where all final vertices and all edges entering them are removed.
Hence all remaining vertices $y$ in $\sat(\pcaf(\phi))$ are such that
$\larc(y)$ is decreased exactly by $1$ ($\earc(y)$ is unchanged), and 
$V(\pcaf(\phi))=V(\phi)\setminus F_0(\phi)$.
We deduce that $F_n(\pcaf(\phi))=F_{n+1}(\phi)$.
Similarly $I_n(\pcai(\psi))=I_{n+1}(\psi)$.
Hence, the above sequence $\phi^n$ satisfies $F_0(\phi^n)=F_n(\phi^0)$,
thus
\[V(\phi^n)=V(\phi^{n-1})\setminus F_0(\phi^{n-1})
=V(\phi^0)\setminus (F_0(\phi^0)\cup\cdots\cup F_{n-1}(\phi^0))\]
and $V(\phi)=\cap_{n\in\N} V(\phi^n)=F_{\infty}(\phi^0)$.
By a similar argument, we get that $V(\psi)=I_{\infty}(\phi)=
F_{\infty}(\phi^0)\cap I_{\infty}(\phi^0)=\emptyset$,
which completes the proof of the theorem.
\end{proof}

\def\cprime{$'$}
\providecommand{\bysame}{\leavevmode\hbox to3em{\hrulefill}\thinspace}
\providecommand{\MR}{\relax\ifhmode\unskip\space\fi MR }
\providecommand{\MRhref}[2]{%
  \href{http://www.ams.org/mathscinet-getitem?mr=#1}{#2}
}
\providecommand{\href}[2]{#2}

\end{document}